\documentclass[a4paper,10pt]{article}
\usepackage{latexsym,amssymb,amsfonts,amsmath,amsthm,mathrsfs,amstext,color,graphicx,times}
\usepackage[mathscr]{euscript}
\usepackage{indentfirst}
\usepackage{cite}
\usepackage{wrapfig}
\usepackage[all]{xy}
\usepackage{tikz}
\usepackage{hyperref}
\usepackage{color}
\usetikzlibrary{patterns}

\newcommand{\R}{\mathbb{R}}

\newtheorem{definition}{Definition}[section]
\newtheorem{theorem}{Theorem}[section]
\newtheorem{corollary}[theorem]{Corollary}
\newtheorem{proposition}{Proposition}[section]
\newtheorem{remark}{Remark}
\newtheorem{lemma}{Lemma}[section]

\newcommand{\tos}{\rightrightarrows} % point-to-set mappings
\DeclareMathOperator{\co}{co}

\DeclareMathOperator{\qep}{QEP}
\DeclareMathOperator{\qvi}{QVI}
\DeclareMathOperator{\mqep}{MQEP}
\DeclareMathOperator{\fix}{Fix}
\DeclareMathOperator{\gra}{graph}
\DeclareMathOperator{\Int}{int}
\title{Coerciveness condition for quasi-equilibrium problems}
\author{John Cotrina
	\thanks{Universidad del Pac\'ifico. Lima, Per\'u. Email: 
	\texttt{ cotrina\_je@up.edu.pe}}
	\and Abderrahim Hantoute 
	\thanks{Center for Mathematical Modeling (CMM), Chile. Email: \texttt{ahantoute@dim.uchile.cl,  asvensson@dim.uchile.cl}}
	\and Anton Svensson\footnotemark[2]
}
\begin{document}
\maketitle

\begin{abstract}
A quasi-equilibrium problem is an equilibrium problem where the constraint set does depend on the reference point. It generalizes important problems such as quasi-variational inequalities and generalized Nash equilibrium problems. We study the existence of equilibria on unbounded sets under a coerciveness condition adapted from one specific for 
quasi-variational inequalities recently proposed by Aussel and Sultana. 
We discuss the relation of our results with others that are present in the literature.
\bigskip


\bigskip

\end{abstract}

\section{Introduction}
By \emph{equilibrium problem}, Blum and Oettli \cite{OB93}, mean the problem of finding:
\begin{align}\label{EP}
x_0\in C\mbox{ such that }f(x_0,y) \geq0\mbox{ for all }y\in C,
\end{align}
where a subset $C$ of $\R^n$ and a function $f:\R^n\times\R^n\to\R$ are given.

Classical existence results for this problem on an unbounded constraint set usually involve the same  
hypotheses as for bounded constraint set together with a coerciveness condition, see for instance \cite{BP01,BP05,JCYG}, and their references.

We consider next the  \emph{quasi-equilibrium problem},
(QEP) for short, which is the object of our main interest in this paper. The QEP consists of an equilibrium problem in which the constraint set depends on the currently analysed point.
More precisely, given a function $f:\R^n\times\R^n\to\R$ and a set-valued map $K:C\tos C$, where
$C$ is a non-empty subset of $\R^n$, the QEP consists of finding:
\begin{align}\label{QEP}
x_0\in C \mbox{ such that }x_0\in K(x_0), \mbox{ and }f(x_0,y)\geq0 \mbox{ for all }y\in K(x_0).
\end{align}
In a similar way, the \emph{Minty quasi-equilibrium problem}, (MQEP) for short, consists of finding:
\begin{align}\label{MQEP}
x_0\in C \mbox{ such that } x_0\in K(x_0),\mbox{ and }f(y,x_0)\leq0 \mbox{ for all }y\in K(x_0).
\end{align}

Recently, these quasi-equilibrium problems have begun to gain more and more attention
due to the fact that they summarise many problems such as quasi-variational inequalities,
generalized Nash equilibrium problems, among others (e.g. \cite{castellani_Giuli15,JC-JZ2}).

Many existence results for (QEP) involve the compactness of $C$, see for instance \cite{ACI,castellani_Giuli15,JC-JZ2,castellani2018,Cubio,Lassonde-b,JC-JZ-2018}. 
In this compact context, Lassonde, in \cite{Lassonde-b}, proposed an existence result without any lower semi-continuity
assumption on the constraint set-valued map. On the other hand, Cubiotti in \cite{Cubio} gave a version
without upper semi-continuity assumption of the constraints. Later in \cite{castellani2018,JC-JZ-2018}, the authors 
used Cubiotti's idea, in order to generalize the famous minimax inequality due to Ky Fan.
Although in \cite{JC-JZ-2018} the authors deal with a non-compact set $C$, they still consider constraint maps
having compact values.

The pioneering works treating the non-compact case were presented by Tian and Zhou in \cite{Tian-Zhou}, Noor and Oettli in  \cite{Oettli}, and Ding in \cite{Ding}.
It is important to remark that in \cite{Tian-Zhou} the authors worked without lower semi-continuity, but they need the upper semi-continuity of the function associated with the quasi-equilibrium problem. In \cite{Oettli,Ding} the continuity of both the function and the constraint map were considered.

Our aim in this work is to provide some existence results for (QEP), under a coerciveness condition which is inspired from \cite{Au-Su}. In Section \ref{preliminaries} we present basic and classical notions on generalized convexity, generalized monotonicity, continuity for set-valued maps and some results. In Section \ref{main} we introduce the coerciveness condition for a quasi-equilibrium problem and we present our main results. Finally, in Section \ref{applications}
we consider some applications on quasi-variational inequality problems and generalized Nash equilibrium problems.

\section{Preliminaries and basic results}\label{preliminaries}

Let $S$ be a subset of $\R^n$. The convex hull and the closure of $S$ will be denoted by $\co(S)$ and
$\overline{S}$, respectively. We denote the open ball and the closed ball in
$\R^n$ with centre $0$ and radius $\varepsilon>0$ by $B_\varepsilon$ and $\overline{B}_\varepsilon$, respectively.

Let us now recall some classical definitions of generalized convexity. A real-valued function $h:\R^n\to\R$ is said to be 
\begin{itemize}
\item \emph{convex} if, for any $x,y\in\R^n$ and $t\in[0,1]$, we have
\[
h(tx+(1-t)y)\leq th(x)+(1-t)h(y);
\]

\item \emph{quasi-convex} if, for any $x,y\in \R^n$ and $t\in[0,1]$,
we have
\[
h(tx+(1-t)y)\leq \max\{h(x),h(y)\};
\]
\item \emph{semi-strictly quasi-convex} if, it is quasi-convex and, for any $x,y\in \R^n$ such that $h(x)\neq h(y)$ 
the following holds
\[
h(tx+(1-t)y)< \max\{h(x),h(y)\}\mbox{ for all }t\in]0,1[.
\]
\end{itemize}
It is clear that every convex function is semi-strictly quasi-convex.
A relevant and useful characterisation of quasi-convexity is that a function is quasi-convex if and only if, its  sub-level sets are convex.
A good reference for quasi-convex functions and quasi-convex optimisation is \cite{Aussel-book}.

Let $K:X\tos Y$ be a set-valued map with $X$ and $Y$ two topological spaces.
The map $K$ is called:
\begin{itemize}
 \item \emph{closed} if its graph is a closed subset of $X\times Y$,
 \item \emph{lower semi-continuous at $x_0$} if for each open set $V$ such that $K(x_0)\cap V\neq\emptyset$
 there exists a neighbourhood $U$ of $x_0$ such that $K(x)\cap V\neq\emptyset$ for all $x\in U$,
 
 \item \emph{upper semi-continuous} at $x_0$ if for any neighbourhood $V$ of $K(x_0)$, 
 there exists a neighbourhood $U$ of $x_0$ such that $K(U):=\left\{y: \exists u\in U, y\in K(u)\right\} \subset V$.
 \end{itemize}
The usual definition of lower semi-continuity of a set-valued map using sequences/nets is equivalent to the one given here using open sets (see for instance Proposition 2.5.6 in \cite{VMPOS}).

We now present some basic results about the lower semi-continuity of certain construction of set-valued maps. First we state three new lemmas, and then we recall some other basic results from the literature that we will use in the following sections.

\begin{lemma}\label{L1}
Let $X,Y$ be two topological spaces, $T:X\tos Y$ be a set-valued map, $x_0\in X$ and $V$ be an open subset of $Y$. If $T$ is lower semi-continuous at $x_0$, then the set valued map $T_V:X\tos Y$ defined by
\[
T_V(x):=T(x)\cap V,\quad x\in X
\]
is lower semi-continuous at $x_0$.
\end{lemma}
\begin{proof}
Let $V_1$ be an open subset of $Y$ such that $T_V(x_0)\cap V_1\neq\emptyset$. We put $V_2:=V_1\cap V$, which is open. Since $T_V(x_0)\cap V_1=T(x_0)\cap V_2$, by lower semi-continuity of $T$, there exists a neighbourhood $U$ of $x_0$ such that $T(x)\cap V_2\neq\emptyset$ for all $x\in U$, or equivalently
$T_V(x_0)\cap V_1\neq\emptyset$ for all $x\in U$, so that $T_V$ is lower semi-continuous at $x_0$.
\end{proof}
%Another important result is stated bellow.
\begin{lemma}\label{L2}
Let $X,Y$ be two topological spaces, $T:X\tos Y$ be a set-valued map and $x_0\in X$. If $T$ is lower semi-continuous at $x_0$, then any set-valued map $S:X\tos Y$ satisfying
\[
T(x)\subset S(x)\subset \overline{T(x)}, \quad \forall x\in X
\]
is lower semi-continuous at $x_0$.
\end{lemma}
\begin{proof}
Let $V$ be an open subset of $Y$ such that $S(x_0)\cap V\neq \emptyset$. Clearly, 
$\overline{T(x_0)}\cap V\neq\emptyset$. Thus the set $T(x_0)\cap V$ is non-empty. Now, by lower semi-continuity of $T$, there exists a neighbourhood $U$ of $x_0$ such that 
$\emptyset\neq T(x)\cap V\subset S(x)\cap V$, for all $x\in U$.
\end{proof}

\begin{lemma}\label{new-lemma}
	Let $T:X\tos Y$ be a set-valued map between a topological space $X$ and a topological vector space $Y$, and $V$ an open and convex subset of $Y$. If $T$ is lower semi-continuous with convex values and $T(x)\cap V\neq\emptyset$, then the set-valued map $T_{\overline{V}}$ given by $T_{\overline{V}}(x):=T(x)\cap \overline{V}$ is lower semi-continuous.
\end{lemma}
\begin{proof}
First, by Lemma 2.1 we know that the set-valued map $S: X\tos Y$ defined by $S(x):=T(x)\cap V$, is lower semi-continuous. 
Second, we observe that $S(x)\subset T_{\overline{V}}\subset \overline{S(x)}$ for each $x\in X$. In fact, to prove the second inclusion (the first being trivial), take $y\in T_{\overline{V}}(x)$. Take also a point $y_0\in S(x)$. By convexity we have that the interval $]y_0,y[$ is included in $S(x)$, and then obviously $y\in \overline{S(x)}$, which prove the inclusion. Finally, we conclude by applying Lemma 2.2 that $T_{\overline{V}}$ is lower semi-continuous.
\end{proof}
\begin{remark}
	Lemma \ref{new-lemma} is a slight refinement of Lemma 1 in \cite{Au-Su}, mainly because it drops two assumptions: the values of $T$ need not to be closed, and the values of $T_{\overline{V}}$ need not to have non-empty interior.
\end{remark}

\begin{lemma}[Lemma 2.3 in \cite{Naselli}]\label{L3}
Let $X,Y$ be two topological spaces and $A$ be a closed subset of $X$. Consider two lower semi-continuous set-valued maps $T:X\tos Y$, $S:A\tos Y$ such that, for every $x\in A$, one has
$S(x)\subset T(x)$. Let $J:X\tos Y$ be defined as
\[
J(x):=\left\lbrace\begin{array}{cc}
T(x),&\mbox{if }x\in X\setminus A\\
S(x),&\mbox{if }x\in A.
\end{array}\right.
\]
Then, the set-valued map $J$ is lower semi-continuous.
\end{lemma}

The following result is part of Theorem 5.9 in \cite{RW}.
\begin{lemma}\label{L4}
Let $T:\R^n\tos\R^m$ be a set-valued map. If $T$ is lower semi-continuous, then so is the set-valued map 
$\co(T):\R^n\tos\R^m$ defined by
\[
\co(T)(x):=\co(T(x)).
\]
\end{lemma}

Given a set-valued map $T:X\tos X$, a point $x\in X$ is said to be a \emph{fixed point} of $T$ if, $x\in T(x)$. We denote $\fix(T)$ the set of fixed points of $T$. We state below Himmelberg's well-known fixed point theorem. 
\begin{theorem}[Theorem 2 in \cite{Himmelberg}]\label{FPT}
Let $A$ be a non-empty and convex subset of a Hausdorff, locally convex topological vector space $Y$, and let $T:A\tos A$ be a set-valued map. If $T$ is upper semi-continuous with convex, closed and non-empty values, and $T(A)$ is contained in some compact subset $N$ of $A$, then $\fix(T)$ is a non-empty set.
\end{theorem}
As a consequence of Himmelberg's fixed point theorem and one of Michael's famous selection theorems we obtain the following proposition.
\begin{proposition}[Corollary 1 in \cite{JC-JZ-2018}]\label{JoJa}
Given a non-empty, convex and closed subset $C$ of $\R^n$, if $T:C\tos C$ is lower semi-continuous with non-empty and convex values, and $T(C)$ is bounded; then $\fix(T)$ is a non-empty set.
\end{proposition}

Given a set-valued map $T:X\tos Y$ and given $y\in Y$, the \emph{fibre} of $T$ at $y$ is the set
\[
T^{-1}(y):=\{x\in X:~y\in T(x)\}.
\]
The following result corresponds to Theorem 4 of Chapter 5 in \cite{Lassonde-b}, but here the open graph is replaced by open fibres and the proof if the same.
\begin{theorem}\label{Lassonde-criteria}
Let $C$ be a compact, convex and non-empty subset of a locally convex topological vector space and
$S,T:C\tos C$ be two set-valued maps. If the following assumptions hold
\begin{enumerate}
\item $S$ is upper semi-continuous with convex, compact and non-empty values,
\item $T$ is convex-valued with open fibres and $\fix(T)=\emptyset$,
\item the set $\{x\in C:~S(x)\cap T(x)\neq\emptyset\}$ is open;
\end{enumerate}
then there exists $x\in\fix(S)$ such that $S(x)\cap T(x)=\emptyset$.
\end{theorem}

We now recall some different notions of \emph{generalized monotonicity}. 

A function $f:\R^n\times \R^n\rightarrow \R$, is said to be:
\begin{itemize}
\item \emph{pseudo-monotone} on a subset $C$ of $\R^n$ if, for all $x,y\in C$ the following implication holds
\[
f(x,y)\geq0\Rightarrow f(y,x)\leq0;
\]
\item \emph{quasi-monotone} on a subset $C$ of $\R^n$ if, for all $x,y\in C$ the following implication holds
\[
f(x,y)>0\Rightarrow f(y,x)\leq0;
\]
\item \emph{properly quasi-monotone} on a convex subset $C$ of $\R^n$ if
for all $x_1,\dots, x_m\in C$, and all $x\in \co(\{ x_1,\dots,x_m\} )$
\[
\min_{i=1,\dots,m}f(x_i,x)\leq0.
\]
\end{itemize}

These notions of generalized monotonicity for functions are inspired from some corresponding ones for set-valued maps. 
Nevertheless, it is important to note that some of the relations between notions of generalized monotonicity for set-valued maps are not longer true for functions. In fact, neither pseudo-monotonicity implies proper quasi-monotonicity, nor the last one implies quasi-monotonicity in general (see the examples in \cite{BP01}). 

Another important concept that we will consider in this paper is the upper sign property for functions.
\begin{itemize}
\item A function $f:\R^n\times\R^n\to\R$ is said to have the 
\emph{upper sign property} on $C$ if for all $x,y\in C$ 
the following implication holds:
\begin{align}\label{upper}
\bigl(
f(x_t,x)\leq0,~
\forall~ t\in\,]0,1[~
\bigr)
\Rightarrow ~ f(x,y) \geq0.
\end{align}
\end{itemize}

Recently in \cite{JC-JZ-2018}, the authors showed that under suitable assumptions the upper sign property of $f$ is equivalent to the pseudo-monotonicity of $-f$. 

%============
\section{Main results}\label{main}

Let $f:\R^n\times \R^n\to\R$ be a function, $C$ be a non-empty subset of $\R^n$ and  $K:C\tos C$ be a set-valued map. 

\begin{definition}
The quasi-equilibrium problem associated to $f$ and $K$ is said to satisfy the \emph{uniform coerciveness condition} (UCC) if
the following two conditions hold:
\begin{enumerate}
 \item there exists $\rho>0$ such that $K(w)\cap B_\rho\neq\emptyset$, for all $w\in C$,
 \item for each $z\in \fix(K)$, there exists $\rho_z\in]0,\rho[$ such that
\begin{align*}
\forall x\in K(z)\cap\overline{B}_\rho\setminus \overline{B}_{\rho_z},~\exists y\in K(z) \mbox{ with }\|y\|<\|x\|
\mbox{ and }f(x,y)\leq0.
\end{align*}
\end{enumerate}
The positive number $\rho$ will be called \emph{coercive radius} of the quasi-equilibrium problem.
\end{definition}

We now denote by $\qep(f,K)$ and $\mqep(f,K)$ the solution sets of problems \eqref{QEP} and \eqref{MQEP}, respectively.

In the spirit of Proposition 1 in \cite{Au-Su}, we will show that the second condition of UCC holds for each element of $\qep(f,K)$ under generalized monotonicity.
\begin{proposition}\label{pseudo-coercive}
Let $C$ be a non-empty subset of $\R^n$, $K:C\tos C$ be a set-valued map and $f:\R^n\times\R^n\to\R$ be a function. The following hold:
\begin{enumerate}
\item If $f$ is pseudo-monotone , then for all $z\in \qep(f,K)$ there exists $\rho_z>0$ such that
\begin{align*}
\forall x\in K(z)\setminus \overline{B}_{\rho_z},~\exists y\in K(z) \mbox{ with }\|y\|<\|x\|
\mbox{ and }f(x,y)\leq0.
\end{align*} 
\item If the map $K$ is convex-valued and $f$ has the upper sign property, then for all $z\in\mqep(f,K)$ there exists $\rho_z>0$ such that
\begin{align*}
\forall x\in K(z)\setminus \overline{B}_{\rho_z},~\exists y\in K(z) \mbox{ with }\|y\|<\|x\|
\mbox{ and }f(y,x)\geq0.
\end{align*} 
\end{enumerate}
\end{proposition}
\begin{proof}
The first case is a straightforward adaptation of Proposition 1 in \cite{Au-Su}.

In the second case, let $z$ be an element of $\mqep(f,K)$ and take $\rho_z>\|z\|$. Then for all $x\in K(z)\setminus \overline{B}_{\rho_z}$ and $t\in]0,1[$ we have $tz+(1-t)x\in K(z)$ and thus $f(tz+(1-t)x,z)\leq 0$. We conclude that $f(z,x)\geq0$ due to the upper sign property of $f$, and the result follows from taking $y=z$.
\end{proof}

The following proposition provides conditions under which any solution of the quasi-equilibrium problem on the bounded set $\overline{B}_\rho$ is also solution of the unbounded quasi-equilibrium problem. It is an extension of Lemma 2.2 in \cite{BP05}.
\begin{proposition}\label{Pro-1}
Let $f:\R^n\times\R^n\to\R$  be a function, $C$ be a non-empty subset of $\R^n$ 
and $K:C\tos C$ be a set-valued map. Assume that the quasi-equilibrium problem associated to $f$ and $K$ 
satisfies the UCC
with coercive radius $\rho$ , and the following conditions hold
\begin{enumerate}
\item for each $x\in \fix(K)$, the following implication holds
\[
(f(x,y)\leq0~\wedge~f(x,z)<0)~\to~f(x,ty+(1-t)z)<0,\mbox{ for all }t\in]0,1[,
\]
\item $f(x,x)=0$, for all $x\in\fix(K)$,
\item $K(x)$ is convex, for all $x\in \fix(K)$.
\end{enumerate}
Then, any solution $x_0$ of $\qep(f,K_\rho)$ is also a solution of $\qep(f,K)$.
\end{proposition}
\begin{proof}
Let $x_0\in \qep(f,K_\rho)$. If $\|x_0\|=\rho$, it is clear that $x_0\in K(x_0)\setminus \overline{B}_{\rho_{x_0}}$. By the UCC,
there exists $z\in K(x_0)$ with $\|z\|<\|x_0\|=\rho$ such that $f(x_0,z)\leq0$. 
If there exists $y\in K(x_0)$ such that $f(x_0,y)<0$, then 
by assumption 1 we have  $f(x_0,ty+(1-t)z)<0$ for all $t\in]0,1[$. 
However, we get a contradiction with the fact that $x_0\in \qep(f,K_\rho)$, since if we take $t$ small enough, the point $ty+(1-t)z$
is an element of $ K(x_0)\cap \overline{B}_\rho$.

Now assume that $\|x_0\|<\rho$. If there exists $y\in K(x_0)\setminus \overline{B}_\rho$ such that $f(x_0,y)<0$, we repeat the previous steps with $x_0$ in place of $z$ and we again obtain a contradiction.
\end{proof}
\begin{remark}
It is clear that assumption 1 in the previous result is fulfilled when $f$ is semi-strictly quasi-convex in its second argument. 
\end{remark}
We are ready for our main result without upper semi-continuity assumption.
\begin{theorem}\label{main-result}
Let $f:\R^n\times\R^n\to\R$  be a function, $C$ be a non-empty, convex and closed subset of $\R^n$ 
and $K:C\tos C$ be a set-valued map. If the quasi-equilibrium problem associated to $f$ and $K$ satisfies the UCC
with coercive radius $\rho$ and  the following properties hold:
\begin{itemize}
\item the map $K$ is lower semi-continuous with convex and non-empty values,
\item the set $\fix(K)$ is closed,
\item for each $x\in \fix(K)$, the following implication holds
\[
(f(x,y)\leq0~\wedge~f(x,z)<0)~\to~f(x,ty+(1-t)z)<0,\mbox{ for all }t\in]0,1[;
\]
\end{itemize}
then the set $\qep(f,K)$ is non-empty provided that one of the following assumptions holds
\begin{enumerate}
\item $f$ is properly quasi-monotone 
and has the upper sign property on $C$, and 
the set-valued map $G:\fix(K)\tos C$ defined as
\[
G(x):=\{y\in K(x)\cap \overline{B}_\rho:~f(y,x)>0\}
\]
is lower semi-continuous.
\item $f(x,x)\geq 0$, for all $x\in \fix(K)$, and the set-valued map $R:\fix(K)\tos C$ defined as
\[
R(x):=\{y\in K(x):~f(x,y)<0\}
\]
is lower semi-continuous with convex values.
\end{enumerate}
\end{theorem}
\begin{proof}
The set-valued map $K_\rho:C\tos C$ defined by 
\[
K_\rho(x)=K(x)\cap \overline{B}_\rho,
\]
for all $x\in C$, is lower semi-continuous, due to Lemma \ref{new-lemma}, and has convex and non-empty values.
By Proposition \ref{JoJa}, there exists a fixed point of $K_\rho$. Moreover, it is clear that 
$\fix(K_\rho)=\fix(K)\cap\overline{B}_\rho$. 

\begin{enumerate}
\item In the first case, we define the set-valued map $J_1:C\tos C$ by
\[
J_1(x):=\left\lbrace \begin{array}{cc}
K_\rho(x),&x\in C\setminus \fix(K_\rho)\\
\co(G(x)),&x\in \fix(K_\rho)
\end{array}\right.
\]
which is lower semi-continuous due to Lemmas \ref{L3} and \ref{L4}. Since $K_\rho(C)$ is relatively compact, if $J_1$ is non-empty valued, then by Proposition \ref{JoJa} there exists $x_0\in \fix(J_1)$, that means $x_0\in \fix(K_\rho)$ and there exists $x_1,\dots,x_m\in G(x_0)$ such that $x_0\in\co(\{x_1,\dots,x_m\})$. However, this is a contradiction with the fact that $f$ is properly quasi-monotone. Hence, there exists $x_0\in C$ such that
$J_1(x_0)=\emptyset$, which in turn implies that $x_0\in \fix(K_\rho)$. Clearly, 
$x_0\in \mqep(f,K_\rho)$. 
Thus, since $f$ has the upper sign property we have that $x_0\in\qep(f,K_\rho)$, due to Proposition 3.1 in \cite{ACI}.
\item Now, in the second case, we consider $R_\rho:\fix(K_\rho)\tos C$ defined as
\[
R_\rho(x):=R(x)\cap \overline{B}_\rho,
\]
which is lower semi-continuous with convex values. Thus, the set-valued map $J_2:C\tos C$ defined as
\[
J_2(x):=\left\lbrace\begin{array}{cc}
K_\rho(x),&x\in C\setminus\fix(K_\rho)\\
R_\rho(x),&x\in \fix(K_\rho)
\end{array}\right.
\]
is lower semi-continuous with convex values. If $J_2$ is non-empty valued, then again by Proposition \ref{JoJa} there exists $x_0\in J_2(x_0)$, that means $x_0\in\fix(K_\rho)$ and $x_0\in R_\rho(x_0)$, which in turn implies $f(x_0,x_0)<0$. So, we get a contradiction. Hence, there exists $x_0\in C$ such that $J_2(x_0)=\emptyset$. Thus, $x_0\in\fix(K_\rho)$ and $R_\rho(x_0)=\emptyset$, i.e. 
\[
f(x_0,y)\geq0,
\]
for all $y\in K_\rho(x_0)$.
\end{enumerate}
Finally, in both cases, the result follows from Proposition \ref{Pro-1}.
\end{proof}

\begin{remark}
A few remarks are needed about the previous result.
\begin{enumerate}
\item It is clear that every function $f:\R^n\times\R^n\to\R$ which is properly quasi-monotone and has the upper sign property vanishes on the diagonal of $\R^n\times\R^n$.
\item The lower semi-continuity of $R$ could be deduced from the upper semi-continuity  of $f$ in its second argument and the lower semi-continuity of $K$. Moreover, $R$ is convex valued provided that $f$ is quasi-convex in its second variable.
\end{enumerate}
\end{remark}

It its clear that the UCC holds provided the compactness of $C$. Thus, as a direct consequence of Theorem \ref{main-result}, we recover as corollaries the following three results from the literature. 

\begin{corollary}[Proposition 2.1 in \cite{BP05}]
Let $C$ be a non-empty, compact and convex subset of $\R^n$ and $f:\R^n\times\R^n\to\R$ be a function satisfying the following assumptions
\begin{enumerate}
\item $f$ is properly quasi-monotone,
\item for each $y\in C$, the function $f(\cdot,y)$ is upper sign continuous, which means that it satisfies the following implication
\[
\inf_{t\in]0,1[}f(tx+(1-t)y,y)\geq0 ~\to~f(x,y)\geq0, 
\]
\item for each $x\in C$, the set $\{y\in C:~f(x,y)\leq0\}$ is closed,
\item $f$ is quasi-convex with respect to its second argument,
\item $f$ vanishes on the diagonal of $C\times C$,
\item  the following implication holds
\[
(f(x,y)=0~\wedge~f(x,z)<0)~\to~f(x,ty+(1-t)z)<0,\mbox{ for all }t\in]0,1[.
\]
\end{enumerate}
Then the equilibrium problem associated to $f$ and $C$ admits at least a solution.
\end{corollary}

\begin{proof}
It is enough to show that all assumptions of first case of Theorem \ref{main-result} are satisfied. In that sense,
we consider $K:C\tos C$ a constant set-valued map, defined as $K(x)=C$ for all $x\in C$.
Clearly  $K$ is lower semi-continuous with convex and non-empty values, and $\fix(K)=C$.
By  assumptions 2, 4, 5 and 6,
the upper sign property follows from  Lemma 3 in \cite{castellani2012} and Lemma 2.1 in \cite{ACI}.
Since $C$ is compact, the set valued $G:C\tos C$ in Theorem \ref{main-result} is given by
\[
G(x)=\{y\in C:~f(y,x)>0\}.
\]
Finally, $G$ has open fibres due to assumption 3, which in turn implies that it is lower semi-continuous.
\end{proof}

\begin{corollary}[Theorem 4.5 in \cite{ACI}]
Let $f:\R^n\times\R^n\to\R$ be a function, $C$ be a convex, compact and non-empty subset of $\R^n$,
and $K:C\tos C$ be a set-valued map. Suppose that the following properties hold
\begin{enumerate}
\item  $K$ is closed and lower semi-continuous with convex values, and $\Int(K(x))\neq\emptyset$, for all $x\in C$;
\item $f$ is properly quasi-monotone;
\item $f$ is semi-strictly quasi-convex and lower semi-continuous with respect to its second argument;
\item for all $x,y\in\R^n$ and all sequence $(y_k)_k\subset\R^n$ converging to $y$, the following implication holds
\[
\liminf_{k\to+\infty}f(y_k,x)\leq0~\Rightarrow~ f(y,x)\leq0,
\]
\item $f$ has the upper sign property.
\end{enumerate}
Then, the quasi-equilibrium problem admits a solution.
\end{corollary}
\begin{proof}
Since $C$ is compact, the set-valued map $G$ in Theorem \ref{main-result} is given by
\[
G(x)=\{y\in K(x):~f(y,x)>0\}
\]
for every $x\in \fix(K)$. The lower semi-continuity of $G$ follows from the same steps of the proof of Corollary 7 in \cite{JC-JZ-2018}. The result follows from Theorem \ref{main-result}.
\end{proof}

\begin{corollary}[Theorem 3 in \cite{Oettli}]
Let $C$ be a compact, convex and non-empty subset of $\R^n$, $K:C\tos\R^n$ and $K_C:C\tos C$ be two set-valued maps such that $K_C(x)=K(x)\cap C$, and $f:C\times C\to\R$ be a function. If the following assumptions hold
\begin{enumerate}
\item $K_C$ is upper and lower semi-continuous with convex, compact and non-empty values,
\item $f$ is continuous and $f(x,\cdot)$ is convex, for all $x\in C$,
\item $f(x,x)\geq0$, for all $x\in C$,
\item for each $x\in\fix(K_C)$ there exists $y\in K_C(x)$ such that $f(x,y)\leq0$ and
$]y,z]\cap K_C(x)\neq\emptyset$, for all $z\in K(x)\setminus K_C(x)$;
\end{enumerate}
then the quasi-equilibrium problem associated to $K$ and $f$ admits at least a solution.
\end{corollary}
\begin{proof}
First, note that the fourth assumption implies $\qep(f,K_C)\subset \qep(f,K)$.
The compactness of $C$ implies that the set $\qep(f,K_C)$ is non-empty, due to the second case of Theorem \ref{main-result}, and the result follows.
\end{proof}

Now we state a result without lower semi-continuity assumption.
We can consider it as the non-compact version of Theorem 5 in \cite{Lassonde-b}, in the finite dimensional setting. 

\begin{theorem}\label{Gen-Lassonde}
Let $C$ be a convex, closed and non-empty subset of $\R^n$, $K:C\tos C$ be a set-valued map and 
$f:\R^n\times \R^n\to\R$ be a function. If the quasi-equilibrium problem associated to $f$ and $K$ satisfies the UCC
with coercive radius $\rho$ and the following properties hold:
\begin{enumerate}
\item $K$ is closed with convex and non-empty values,
\item $f(\cdot,y)$ is upper semi-continuous, for all $y\in C$,
\item $f(x,\cdot)$ is quasi-convex, for all $x\in C$,
\item the set $D=\{x\in C\cap\overline{B}_\rho:~\inf_{y\in K(x)\cap \overline{B}_\rho}f(x,y)<0\}$ is open in $C\cap\overline{B}_\rho$,
\item $f$ vanishes on the diagonal of $C \times C$,
\item for each $x\in \fix(K)$, the following implication holds
\[
(f(x,y)\leq0~\wedge~f(x,z)<0)~\to~f(x,ty+(1-t)z)<0,\mbox{ for all }t\in]0,1[;
\]
\end{enumerate}
then the quasi-equilibrium problem admits at least a solution.
\end{theorem}
\begin{proof}
We denote $C_\rho$ the set $C\cap \overline{B}_\rho$.
The set-valued maps $K_\rho,T:C_\rho\tos C_\rho$ defined as
\[
K_\rho(x):=K(x)\cap \overline{B}_\rho~\mbox{ and }~T(x):=\{y\in K_\rho(x):~f(x,y)<0\}.
\]
Clearly
$\gra(K_\rho)=\gra(K)\cap (C_\rho\times C_\rho)$ and
$D=\left\lbrace x\in C_\rho:~T(x)\cap K_\rho(x)\neq\emptyset\right\rbrace$. 
By the uniform coerciveness condition and the first assumption imply that $K_\rho$ is upper semi-continuous with convex, compact and non-empty values. Also, assumptions 2 and 3 imply that $T$ is convex-valued with open fibres.
Since $f$ vanishes on the diagonal on $C\times C$, we deduce that $\fix(T)=\emptyset$. Hence, by Theorem \ref{Lassonde-criteria} there exists $x\in\fix(K_\rho)$ such that $K_\rho(x)\cap T(x)=\emptyset$, that means $x\in\qep(f,K_\rho)$. 
The result follows from Proposition \ref{Pro-1}.
\end{proof}

The previous result is strongly related to Theorem 3 in \cite{Tian-Zhou}. The two set of conditions differ in two aspects, first the authors in \cite{Tian-Zhou} considered that the function $f$ satisfies that for any $x_1,\dots,x_m\in C$ and any $x=\sum_{i=1}^m \lambda_i x_i$ with  $\sum_{i=1}^m\lambda_i=1$ and $\{\lambda_1,\dots,\lambda_m\}\subset[0,1]$, it holds
\[
\sum_{i=1}^m \lambda_i f(x,x_i)\geq0.
\]
Here, this condition is replaced by the quasi-convexity in the second argument of $f$ and the fact that $f$ vanishes on the diagonal of $C\times C$. Both assumptions are independent in general, in \cite{Zhou-Chen} the authors show some examples about it. The second difference is the coerciveness conditions, in \cite{Tian-Zhou} they considered that there exist a non-empty, compact and convex set $Z\subset C$ and a non-empty set $W\subset Z$ such that $K(W)\subset Z$ and
\begin{enumerate}
\item $K(x)\cap Z\neq\emptyset$, for all $x\in Z$;
\item for each $x\in Z\setminus W$ there exists $y\in K(x)\cap Z$ such that $f(x,y)<0$;
\item $\{x\in Z:~\inf_{y\in K(x)\cap Z}f(x,y)\geq0\}$ is closed.
\end{enumerate}
However, this coerciveness condition is too strong. For instance, consider $C:=[0,+\infty[$ and $K:C\tos C$ defined by $K(x):=[1,+\infty[$ which is closed, Moreover, there does not exist non-empty sets $Z, W\subset C$ with $Z$ convex and compact, such that $K(W)\subset Z$.

%==========================
\section{Applications}\label{applications}
In this section, we consider applications to the study of existence of solutions of two well-known problems: (i) the quasi-variational inequality problem, and (ii) the generalized Nash equilibrium problem.
\subsection{Quasi-variational inequality}
Given  a subset $C$ of $\R^n$ and two set-valued maps $T:\R^n\tos \R^n$ and $K:C\tos C$, the set 
$\qvi(T,K)$
denotes the solution set of the \emph{quasi-variational inequality problem}
\[
\{x\in C:~x\in K(x) \mbox{ and exists }x^*\in T(x)\mbox{ such that }\langle x^*,y-x\rangle\geq0,~\forall y\in K(x)\}. 
\]
The quasi-variational inequality problem associated to $T$ and $K$ is said to satisfy the \emph{uniform coerciveness condition} if
the following two conditions hold:
\begin{enumerate}
 \item there exists $\rho>0$ such that $K(w)\cap B_\rho\neq\emptyset$, for all $w\in C$,
 \item for each $z\in \fix(K)$, there exists $\rho_z\in]0,\rho[$ such that
\begin{align*}
&\forall x\in K(z)\cap\overline{B}_\rho\setminus \overline{B}_{\rho_z},~\exists y\in K(z) \mbox{ with }\|y\|<\|x\|
\mbox{ such that }\\
&\forall x^*\in T(x),~\langle x^*,y-x\rangle\leq0.
\end{align*}
\end{enumerate}
We will call \emph{coercive radius} of the quasi-variational inequality problem the real number $\rho$ in the previous definition.

Now, we consider  the function $f_T:\R^n\times\R^n\to\R$ defined as
\begin{align}\label{qvi-qep}
f_T(x,y):=\sup_{x^*\in T(x)}\langle x^*,y-x\rangle.
\end{align}
It is clear that $\qep(f_T,K)=\qvi(T,K)$, provided $T$ has compact and non-empty values.

\begin{lemma}
The quasi-variational inequality problem associated to $T$ and $K$ satisfies the uniform coerciveness condition
if and only if, the quasi-equilibrium problem associated to $f_T$ (defined in \eqref{qvi-qep}) and $K$ satisfies the UCC.
%uniform coerciveness condition.
\end{lemma}

As a direct consequence of the previous lemma and  Theorem \ref{main-result} we obtain the following result.

\begin{theorem}
Let $C$ be a closed, convex and non-empty subset of $\R^n$, and $T:\R^n\tos\R^n$, $K:C\tos C$ be two set-valued maps. If the quasi-variational inequality satisfies the uniform coerciveness condition with coercive radius $\rho$, then it admits at least a solution provided that
\begin{enumerate}
\item $T$ is compact and non-empty valued,
\item $T$ is properly quasi-monotone on $C$ i.e., for all $x_1,\dots,x_m\in C$ and any  $x\in\co(\{x_1,\dots,x_m\})$, there exists $i$ such that
\[
\langle x_i^*,x-x_i\rangle\leq0,\mbox{ for all }x_i^*\in T(x_i),
\]
\item $T$ is upper sign-continuous on $C$, that means for
all $x,y\in C$, the following implication holds
\[
\left(\forall t\in]0,1[,~\inf_{x_t^*\in T(x_t)}\langle x_t^*,y-x\rangle\geq0\right)~\Rightarrow~
\sup_{x^*\in T(x)}\langle x^*,y-x\rangle\geq0,
\]where $x_t=tx+(1-t)y$.,
\item $K$ is lower semi-continuous with convex and non-empty values
\item The set $\fix(K)$ is closed and the set-valued map $G:\fix(K)\tos C$ defined as
\[
G(x):=\left\lbrace y\in K(x)\cap\overline{B}_\rho:~\sup_{x^*\in T(x)}\langle x^*,y-x\rangle>0\right\rbrace
\]
is lower semi-continuous.
\end{enumerate}
\end{theorem}
\begin{proof}
Clearly $f_T$ is properly quasi-monotone and has the upper sign property. Therefore, the result follows from the fact $\qvi(T,K)=\qep(f_T,K)$ and the first case of Theorem \ref{main-result}.
\end{proof}

\begin{remark}
A few remarks are needed.
\begin{enumerate}
\item The previous result is not a consequence of Theorem 1 in \cite{Au-Su}, because $T$ is properly quasi-monotone and $K$ is not closed. Theorem 3 in \cite{Au-Su} proposes an existence result under quasi-monotonicity, that means for all $(x,x^*)$ and $(y,y^*)$ in the graph of $T$ the following implication holds
\[
\langle x^*,y-x\rangle>0~\Rightarrow~\langle y^*,y-x\rangle\geq0, 
\]
but in this case it needs more regularity assumptions on the constraint map.
\item The fourth assumption in the previous result holds, for instance, when the map $K$ is closed and the set 
\[
\left\lbrace (x,y)\in C\times C:~\sup_{x^*\in T(x)}\langle x^*,y-x\rangle\leq0\right\rbrace
\]
is closed. 
\end{enumerate}
\end{remark}

\subsection{Generalized Nash equilibrium problem}
A generalized Nash equilibrium problem (GNEP) consists of
 $p$ players. Each player $\nu$ controls the decision variable $x^\nu\in C_\nu$, where $C_\nu$ is a non-empty convex and closed subset
 of $\R^{n_\nu}$. We denote by $x=(x^1,\dots,x^p)\in \prod_{\nu=1}^p C_\nu=C$ the vector formed by all these decision
 variables and by $x^{-\nu}$, we denote the strategy vector of all the players different from player $\nu$. The set of all such
vectors will be denoted by $C^{-\nu}$. We sometimes write  $(x^\nu,x^{-\nu})$ instead of $x$ in order to emphasize the $\nu$-th player's
variables within $x$. Note that this is still the vector $x=(x^1,\dots,x^\nu,\dots,x^p)$, and the notation $(x^\nu,x^{-\nu})$
does not mean that the block components of $x$ are reordered in such a way that $x^\nu$ becomes the first block.
Each player $\nu$ has an objective function $\theta_\nu:C\to\R$ that depends on all player's strategies. Each player's strategy must belong to a set identified by the set-valued map $K_\nu:C^{-\nu}\tos C_\nu$ in the sense that the strategy
space of player $\nu$ is $K_\nu(x^{-\nu})$, which depends on the rival player's strategies $x^{-\nu}$.
Given the strategy $x^{-\nu}$, player $\nu$ chooses a strategy $x^\nu$ such that it solves the following optimisation problem
\begin{align}\label{opt-nu}
\min_{x^\nu} \theta_\nu(x^\nu,x^{-\nu}),\mbox{ subject to }x^\nu\in K_\nu(x^{-\nu}),
\end{align}
for any given strategy vector $x^{-\nu}$ of the rival players. The solution set of problem \eqref{opt-nu}
is denoted by ${\rm Sol}_\nu(x^{-\nu})$. Thus, a {\em generalized Nash equilibrium} is a vector 
$\hat{x}$ such that $\hat{x}^\nu\in {\rm Sol}_\nu(\hat{x}^{-\nu})$, for any $\nu$.

Associated to a GNEP, there is a function $f^{NI}:\R^n\times \R^n\to\R$, defined by
\[
 f^{NI}(x,y):=\sum_{\nu=1}^p\{\theta_\nu(y^\nu,x^{-\nu})-\theta_\nu(x^\nu,x^{-\nu})\},
\]
which is called Nikaid\^o-Isoda function and was introduced in \cite{Nikaido-Isoda}. 
Additionally, we need to consider the set-valued map $K:C\tos C$ defines as 
\[
K(x):=\prod_{\nu=1}^p K_\nu(x^{-\nu}).
\]
\begin{lemma}\label{equiv-2}
A vector $\hat{x}$ is a solution of the GNEP if and only if, $\hat{x}\in\qep(f^{NI},K)$.
\end{lemma}

A GNEP satisfies the \emph{coerciveness condition} if there exists $\rho>0$ such that\begin{enumerate}
\item $K(x)\cap B_\rho\neq\emptyset$, for all $x\in C$;
\item for each $z\in\fix(K)$, there exists $\rho_z\in]0,\rho[$ such that 
\[
\forall x\in K(z)\cap\overline{B}_\rho\setminus \overline{B}_{\rho_z},~\exists y\in K(z)%\cap\overline{B}_\rho
\mbox{ with }\|y\|<\|x\|\mbox{ and }\]
\[\theta_\nu(y^\nu,x^{-\nu})\leq \theta_\nu(x), \mbox{ for each }\nu.
\]
\end{enumerate}
 In this case, the real number $\rho$ is called \emph{coercive radius} of GNEP.
 \begin{lemma}\label{GNEP-coercive}
 If the GNEP satisfies the coerciveness condition, then the quasi-equilibrium problem associated to Nikaid\^o-Isoda function and $K$ satisfies the uniform coerciveness condition.
 \end{lemma}
\begin{proof}
It is enough to see that if for each $\nu$ we have $\theta_\nu(y^\nu,x^{-\nu})\leq\theta_\nu(x)$, then
\[
f^{NI}(x,y)=\sum_{\nu=1}^p \theta_\nu(y^\nu,x^{-\nu})-\theta_\nu(x)\leq0.
\]
\end{proof}
\begin{remark}
When every objective function is differentiable and convex with respect to the own variable $x^\nu$, our coerciveness condition implies the one proposed by Aussel and Sultana in \cite{Au-Su}.
\end{remark}

Thanks to  Lemmas \ref{equiv-2} and \ref{GNEP-coercive}, we have the following result on the existence of solutions of a GNEP, which is a direct consequence of Theorems \ref{main-result} and \ref{Gen-Lassonde}.

\begin{theorem}\label{FGNEP}
For any $\nu \in \{1,2, ...,p \}$, let $C_\nu$ be a non-empty, closed and convex subset of $\R^{n_\nu}$, 
$\theta_\nu:\R^n\to\R$ be a continuous function and $K_\nu:C^{-\nu}\tos C_\nu$ be a set-valued map. If the GNEP satisfies the coerciveness condition with coercive radius $\rho$, then it admits a  solution provided that for each $\nu$, $\theta_\nu$ is convex with respect to the $x^{\nu}$ variable, and one of the following sets of assumptions  holds:
\begin{enumerate}
\item[A1]
\begin{enumerate}
\item The set $\fix(K)$ is closed,
\item  for each $\nu$,  the map $K_\nu$ is lower semi-continuous with non-empty and convex values.
\end{enumerate}

\item[A2]
\begin{enumerate}
\item for each $\nu$,  the map $K_\nu$ is closed with convex and non-empty values,
\item the set
\[
N=\left\lbrace x\in C\cap\overline{B}_\rho:~\inf_{y\in K(x)\cap\overline{B}_\rho}\sum_\nu\theta_\nu(y^\nu,x^{-\nu})<\sum_\nu\theta_\nu(x)\right\rbrace
\]
is open.
\end{enumerate}
\end{enumerate}
\end{theorem}
\begin{proof}
It is clear that Nikaid\^o-Isoda function is continuous and convex in its second argument.
By Lemma \ref{GNEP-coercive}, we have that the quasi-equilibrium problem associated to $f^{NI}$ and $K$ satisfies the UCC. 
In the case $A1$, the map $K$ is lower semi-continuous with convex and non-empty values.
Hence the set-valued map $R$ in the second case of Theorem \ref{main-result} is lower semi-continuous with convex values. So, the result follows from Theorem \ref{main-result} and Lemma \ref{equiv-2}.

Finally, in the case $A2$ the map $K$ is closed with convex and non-empty values.
Moreover, it holds
\[
N=\left\lbrace x\in C\cap\overline{B}_\rho:~\inf_{y\in K(x)\cap\overline{B}_\rho}f^{NI}(x,y)<0\right\rbrace.
\]
Hence, the result follows from Theorem \ref{Gen-Lassonde} and Lemma \ref{equiv-2}.
\end{proof}
The previous result is strongly related to Theorem 5 in \cite{Au-Su}. However, it is important to notice that in the set of assumptions $A1$ the constraint set-valued maps are not closed. For $A2$, the constraint maps are not lower semi-continuous. Moreover, both cases are not differentiable, and it is possible that the images of the constraint maps could have empty interior.

\bibliographystyle{abbrv}

\begin{thebibliography}{99}\rm
\bibitem{Aussel-book}
Aussel D., New developments in Quasiconvex optmization. In
{\em Fixed Point Theory, Variational Analysis and Optimization}, pages 173--208. Taylor \& Francis, 2014.
%
\bibitem{ACI}
Aussel D., Cotrina J. and Iusem A.,
{\em An existence result for quasi-equilibrium problems}.
J. Convex Anal., 24: 55--66, 2017. 
%
\bibitem{Au-Su}
Aussel D. and Sultana A.,
{\em Quasi-variational inequality problems with non-compact valued constraint maps}.
J. Math Anal. Appl., 456 (2): 1482--1494, 2017.
%
\bibitem{BP01}
Bianchi M. and Pini R.,
{\em A note on equilibrium problems with properly quasimonotone bifunctions}.
J. Global Optim., 20: 67--76, 2001.
%
\bibitem{BP05}
Bianchi M. and Pini R.,
{\em Coercivity Conditions for Equilibrium Problems}.
J. Optim Theory Appl, 124(1): 79--92, 2005.
%
\bibitem{OB93}
Blum E. and Oettli W.,
{\em From optimization and variational inequalities to equilibrium problems}.
Math Stud., 63: 1--23, 1993.
%
\bibitem{castellani2012}
Castellani M. and Giuli M.,
{\em Refinements of existence results for relaxed quasimonotone equilibrium problems}.
J. Global Optim., 57: 1213--1227, 2013.
%
\bibitem{castellani_Giuli15}
Castellani M. and Giuli M.,
{\em An existence result for quasiequilibrium problems in separable Banach spaces}.
J. Math Anal. Appl., 425: 85--95, 2015.
%
\bibitem{castellani2018}
Castellani M., Giuli M. and Pappalardo M.,
{\em A Ky Fan minimax inequality for quasiequilibria on finite-dimensional spaces}.
J. Optim Theory Appl, 179: 53--64, 2018.
%
\bibitem{JCYG}
Cotrina J. and Garc\'ia Y.,
{\em Equilibrium problems: Existence results and Applications}.
Set-Valued Var. Anal., 26: 159--177, 2018.
%
\bibitem{JC-JZ2}
Cotrina J. and Z\'u\~niga J.,
{\em A note on quasi-equilibrium problems}.
Ope. Res. Lett., 46(1): 138-- 140, 2018.
%
\bibitem{JC-JZ-2018}
Cotrina J. and Z\'u\~niga J.,
{\em Quasi-equilibrium problems with non-self constraint map}.
J. Global Optim., 2019.
%
\bibitem{Cubio}
Cubiotti P.,
{\em Existence of Nash Equilibria for Generalized Games without Upper Semicontinuity}.
Int J Game Theory, 26: 267--273, 1997.
%
\bibitem{Ding}
Ding X. P.,
{\em Existence of solutions for quasi-equilibrium problems in noncompact topological spaces}.
Computer Math. Appl., (39): 13--21, 2000.
%
\bibitem{VMPOS}
G{\"o}pfert A., Riahi H., Tammer C. and Z\u{a}linescu C., 
{\em Variational Methods in Partially Ordered Spaces}.
CMS Books in Mathematics. Springer, New York, NY, 2003.

%
\bibitem{Himmelberg}
C. J. Himmelberg,
{\em Fixed points of compact multifunctions}.
J. MAth. Anal Appl., 38: 205--207, 1972.

%
\bibitem{Lassonde-b}
Lassonde M.,
{\em Minimax et Point Fixes}.
2005. Unpublished.

%
\bibitem{Nikaido-Isoda}
Nikaid{\^o} H.  and Isoda K.,
{\em Note on noncooperative convex games}.
Pacific J. Math. 52: 807--815, 1955.

%
\bibitem{Naselli}
Ricceri O. N.,
{\em On the covering dimension of the fixed point set of certain multifunctions }.
Comment Math Univ Carolin. 32: 281--286, 1991.

%
\bibitem{Oettli}
Noor M. A. and Oettli W.,
{\em On general nonlinear complementary problems and quasi-equilibria}.
Le Mathematiche, 49(2): 313--331, 1994.

%
\bibitem{RW}
Rockafellar R.T.  and  Wets R.,
{\em Variational Analysis}.
Springer, Berlin, 1997.

%
\bibitem{Tian-Zhou}
Tian G. and Zhou J.,
{\em Quasi-variational inequalities without the concavity assumption}.
J. Math Anal. Appl., 172: 289--299, 1993.
%
\bibitem{Zhou-Chen}
Zhou J.X. and Chen G.,
{\em Diagonal convexity conditions for problems in convex analysis and quasi-variational inequalities}.
J. Math Anal. Appl. , 132: 213--225, 1988
\end{thebibliography}

\end{document}